\documentclass{amsart} 

\usepackage{amsrefs,amsmath,amsthm,amssymb}            
\usepackage[matrix,arrow,curve,frame]{xy}

\usepackage{pstricks}
\newrgbcolor{mygray}{0.7 0.7 0.7}
\newgray{gridline}{0.8}
\newrgbcolor{pcolor}{0 0 0}
\newrgbcolor{NCcolor}{0 0.7 0}
\newrgbcolor{honecolor}{0 0 0}
\newrgbcolor{hzerocolor}{0 0 0}
\newcommand{\cirrad}{0.06}
\newcommand{\labelrad}{0.1}
\newrgbcolor{dotcolor}{0 0 0}

\definecolor{mgreen}{RGB}{0,150,0}
\definecolor{mblue}{RGB}{65,105,225}

\usepackage[colorlinks=true,linkcolor=blue,citecolor=mblue]{hyperref}

\newtheorem{thm}[subsection]{Theorem}
\newtheorem{prop}[subsection]{Proposition}
\newtheorem{lemma}[subsection]{Lemma}

\theoremstyle{definition}  
\newtheorem{ex}[subsection]{Example}

\newcommand{\map}{\rightarrow}
\newcommand{\field}[1]  {\mathbb #1} 
\newcommand{\F}         {\field F}
\newcommand{\R}         {\field R}
\newcommand{\C}         {\field C}

\DeclareMathOperator{\Ext}{Ext}

\newcommand{\nc}[2]{\frac{\gamma}{\rho^{#1} \tau^{#2}}}
\newcommand{\Qc}[2]{\frac{Q}{\rho^{#1} \tau^{#2}}}

\numberwithin{equation}{section} 

\begin{document}

\title{$C_2$-equivariant and $\R$-motivic stable stems, II}

\author{Eva Belmont} 
\address{Department of Mathematics \\Northwestern University \\Evanston, IL
60208}
\email{ebelmont@northwestern.edu}
\author{Bertrand J. Guillou}
\address{Department of Mathematics\\ University of Kentucky\\
Lexington, KY 40506, USA}
\email{bertguillou@uky.edu}
\author{Daniel C.\ Isaksen}
\address{Department of Mathematics \\ Wayne State University\\
Detroit, MI 48202}
\email{isaksen@wayne.edu}
\thanks{The second author was supported by NSF grant DMS-1710379.
The third author was supported by NSF grant DMS-1202213.}

\subjclass[2000]{14F42, 55Q45, 55Q91, 55T15}

\keywords{stable homotopy group, equivariant stable homotopy theory,
motivic stable homotopy theory, Adams spectral sequence}

\begin{abstract}
We show that the $C_2$-equivariant and $\R$-motivic stable
homotopy groups are isomorphic in a range.  This result supersedes
previous work of Dugger and the third author.
\end{abstract}

\maketitle

\section{Introduction}

This article is part of an on-going project to make
explicit computations of stable homotopy groups in 
the $\C$-motivic, $\R$-motivic, $C_2$-equivariant, and classical
stable homotopy theories, as depicted in the diagram
\begin{equation}
\label{eq:comparison}
\xymatrix{
\R\textrm{-motivic} \ar[d]_{\textrm{realization}} \ar[rrr]^{\textrm{extension of scalars}} &  & &
\C\textrm{-motivic} \ar[d]^{\textrm{realization}} \\
 C_2\textrm{-equivariant} \ar[rrr]_{\textrm{forgetful}} & & & 
 \textrm{classical.}
}
\end{equation}
The horizontal arrows labelled ``realization" refer to
the Betti realization functors that 
take a variety over $\C$ (resp., over $\R$)
to the space (resp., $C_2$-equivariant space) of $\C$-valued points.
The vertical arrow labelled ``extension of scalars" refers
to the functor that takes a variety over $\R$ and views it as a
variety over $\C$.
The vertical arrow labelled ``forgetful" refers to the functor
that takes a $C_2$-equivariant object to its underlying non-equivariant
object.

The goal of this article is to study the top horizontal arrow
in Diagram (\ref{eq:comparison}).
We show that there is an isomorphism 
\[
\pi^{\R}_{*,*} \map \pi^{C_2}_{*,*}
\]
in a range of degrees.
Here $\pi^\R_{*,*}$ are the $\R$-motivic stable homotopy groups 
completed at $2$ and $\eta$, and
$\pi^{C_2}_{*,*}$ are the corresponding $C_2$-equivariant 
stable homotopy groups.
For the purposes of this paper,
$\pi^\R_{*,*}$ and $\pi^{C_2}_{*,*}$ can be defined as the targets
of the $\R$-motivic and $C_2$-equivariant Adams spectral sequences; in
particular, these are 2-complete homotopy groups.
The map is induced by equivariant Betti realization
that takes a variety over $\R$ to the space of $\C$-valued
points, equipped with the conjugation action \cite{MV99}*{Section 3.3}, \cite{HO16}*{Section 4.4}. 
In practice, information typically flows from source to target along
the isomorphism.
Even though 
$\pi^\R_{*,*}$ is highly non-trivial \cite{BI19} \cite{DI17},
it is somewhat easier to compute than $\pi^{C_2}_{*,*}$.

See the introduction of \cite{DI17b} for a more thorough discussion of the
objects and categories under consideration.  We assume that the reader
is familiar with the motivic and $C_2$-equivariant Adams spectral sequences.
Relevant details appear in \cite{BI19} \cite{DI10} \cite{GHIR20}
 \cite{Isaksen14c}.

Our work is a natural sequel to the article \cite{DI17b}, which 
establishes an isomorphism between
$\R$-motivic and $C_2$-equivariant stable homotopy groups in a strictly
smaller range.
The method of \cite{DI17b} is to compare cobar complexes, which then
yields a comparison of Adams $E_2$-pages.  In turn, this leads to a 
comparison of stable homotopy groups.  
While the cobar complex has good formal properties, it is a wasteful
construction in the sense that it is much larger than needed to
compute Adams $E_2$-pages.  The approach of this article, in effect,
ignores large parts of the cobar complexes that do not contribute to
Adams $E_2$-pages.

More specifically, we will compare the 
$\R$-motivic and $C_2$-equivariant
$\rho$-Bockstein spectral sequences \cite{BI19} \cite{DI17} \cite{GHIR20}
that converge to the $\R$-motivic and $C_2$-equivariant Adams
$E_2$-pages respectively.  
We will show that these $\rho$-Bockstein spectral sequences are
isomorphic in a range.  As in \cite{DI17b}, this implies
that the Adams $E_2$-pages are isomorphic in a range, which further implies that the 
stable homotopy groups are isomorphic in a range as well.

\begin{thm}
\label{thm:pi-iso}
The Betti realization map
\[
\pi^\R_{s,w} \map \pi^{C_2}_{s,w}
\]
\begin{enumerate}
\item
is an isomorphism if $2w-s < 5$ and $(s,w) \neq (0,2)$.
\item
is an injection if $2w-s = 5$.
\end{enumerate}
\end{thm}

\begin{proof}
Proposition \ref{prop:Adams-iso} shows that
the $\R$-motivic and $C_2$-equivariant Adams $E_\infty$-pages
are isomorphic when $2w-s< 5$ and $s \neq 0$.
In other words, $\pi^\R_{s,w}$ and $\pi^{C_2}_{s,w}$
have isomorphic associated graded objects, so they are isomorphic.
This establishes part (1).

The proof of part (2) is essentially the same.
\end{proof}

Theorem \ref{thm:pi-iso} is stated in terms of the stem $s$ and the
weight $w$.  In some situations, it is more convenient to work with the
stem $s$ and the coweight $s-w$.  In those terms, Theorem \ref{thm:pi-iso}
says that Betti realization:
\begin{enumerate}
\item
is an isomorphism if $s < 2(s-w) + 5$ and $(s,s-w) \neq (0,-2)$.
\item
is an injection if $s = 2(s-w) + 5$.
\end{enumerate}

In order to further illustrate Theorem \ref{thm:pi-iso}, 
Figures \ref{fig:Adams-0}--\ref{fig:Adams-3} show $C_2$-equivariant Adams charts in coweights
$0$ through $3$.  These charts show the range of stems in which
$\R$-motivic and $\C_2$-equivariant stable homotopy groups are 
isomorphic.  The elements in green on the far right of each chart
are the first $C_2$-equivariant classes that have no $\R$-motivic
analogues.  Explanations for the computations in these charts will
appear elsewhere.  Here is a key for reading the charts:
\begin{itemize}
\item
Vertical lines indicate multiplications by $h_0$.
\item
Horizontal lines indicate multiplications by $\rho$.
\item
Lines of slope $1$ indicate multiplications by $h_1$.
\item
Arrows indicate infinite sequences of elements that are related by
multiplications.
\item
Vertical dashed lines indicate hidden $h_0$ extensions.
\item
Dashed lines of negative slope indicate hidden $\rho$ extensions.
\end{itemize}

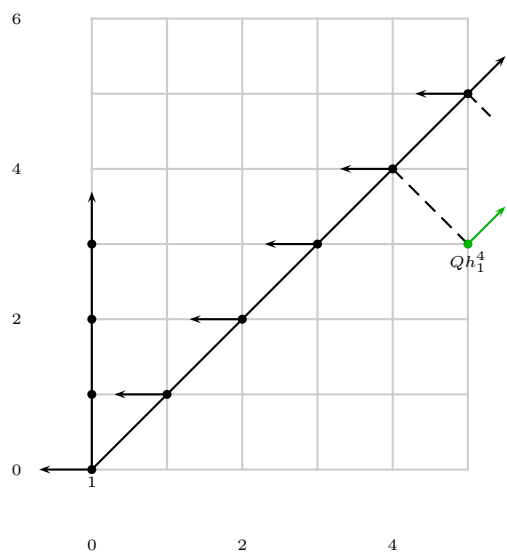
\begin{figure}
\caption{$C_2$-equivariant Adams $E_{\infty}$-page in coweight $0$
\label{fig:Adams-0}}

\begin{pspicture}(-1,-1)(6,7)
\psgrid[unit=1,gridcolor=gridline,subgriddiv=0,gridlabelcolor=white](0,0)(5,6)
\tiny
\rput(0,-1){0}
\rput(2,-1){2}
\rput(4,-1){4}

\rput(-1,0){0}
\rput(-1,2){2}
\rput(-1,4){4}
\rput(-1,6){6}

\psline[linecolor=pcolor,linestyle=dashed](5,3)(4,4)
\pscircle*[linecolor=NCcolor](5,3){\cirrad}
\psline[linecolor=NCcolor]{->}(5,3)(5.50,3.50)
\uput{\labelrad}[-90](5,3){$Q h_1^4$}
\psline[linecolor=pcolor,linestyle=dashed](5.3,4.7)(5,5)

\uput{\labelrad}[-90](0.00,0.00){$1$}
\psline[linecolor=hzerocolor](0.00,0.00)(0.00,1.00)
\psline[linecolor=pcolor]{->}(0.00,0.00)(-0.70,0.00)
\psline[linecolor=honecolor](0.00,0.00)(1.00,1.00)
\pscircle*[linecolor=dotcolor](0.00,0.00){\cirrad}
\psline[linecolor=hzerocolor](0.00,1.00)(0.00,2.00)
\pscircle*[linecolor=dotcolor](0.00,1.00){\cirrad}
\psline[linecolor=hzerocolor](0.00,2.00)(0.00,3.00)
\pscircle*[linecolor=dotcolor](0.00,2.00){\cirrad}
\psline[linecolor=hzerocolor]{->}(0.00,3.00)(0.00,3.70)
\pscircle*[linecolor=dotcolor](0.00,3.00){\cirrad}
\psline[linecolor=pcolor]{->}(1.00,1.00)(0.30,1.00)
\psline[linecolor=honecolor](1.00,1.00)(2.00,2.00)
\pscircle*[linecolor=dotcolor](1.00,1.00){\cirrad}
\psline[linecolor=pcolor]{->}(2.00,2.00)(1.30,2.00)
\psline[linecolor=honecolor](2.00,2.00)(3.00,3.00)
\pscircle*[linecolor=dotcolor](2.00,2.00){\cirrad}
\psline[linecolor=pcolor]{->}(3.00,3.00)(2.30,3.00)
\psline[linecolor=honecolor](3.00,3.00)(4,4)
\pscircle*[linecolor=dotcolor](3.00,3.00){\cirrad}
\psline[linecolor=pcolor]{->}(4.00,4.00)(3.30,4.00)
\psline[linecolor=honecolor](4.00,4.00)(5,5)
\pscircle*[linecolor=dotcolor](4.00,4.00){\cirrad}
\psline[linecolor=pcolor]{->}(5.00,5.00)(4.30,5.00)
\psline[linecolor=honecolor]{->}(5.00,5.00)(5.50,5.50)
\pscircle*[linecolor=dotcolor](5.00,5.00){\cirrad}
\end{pspicture}
\end{figure}

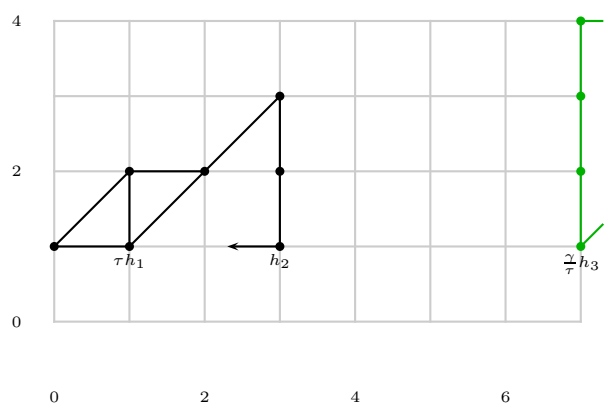
\begin{figure}
\caption{$C_2$-equivariant Adams $E_{\infty}$-page in coweight $1$
\label{fig:Adams-1}}
\begin{pspicture}(-1,-5)(7,5)
\psgrid[unit=1,gridcolor=gridline,subgriddiv=0,gridlabelcolor=white](0,0)(7,4)
\tiny
\rput(0,-1){0}
\rput(2,-1){2}
\rput(4,-1){4}
\rput(6,-1){6}

\rput(-0.5,0){0}
\rput(-0.5,2){2}
\rput(-0.5,4){4}

\uput{\labelrad}[-90](1.00,1.00){$\tau h_1$}
\psline[linecolor=hzerocolor](1.00,1.00)(1.00,2.00)
\psline[linecolor=pcolor](1.00,1.00)(0.00,1.00)
\psline[linecolor=honecolor](1.00,1.00)(2.00,2.00)
\pscircle*[linecolor=dotcolor](1.00,1.00){\cirrad}
\psline[linecolor=honecolor](0.00,1.00)(1.00,2.00)
\pscircle*[linecolor=dotcolor](0.00,1.00){\cirrad}
\psline[linecolor=pcolor](2.00,2.00)(1.00,2.00)
\psline[linecolor=honecolor](2.00,2.00)(3.00,3.00)
\pscircle*[linecolor=dotcolor](2.00,2.00){\cirrad}
\pscircle*[linecolor=dotcolor](1.00,2.00){\cirrad}
\uput{\labelrad}[-90](3.00,1.00){$h_2$}
\psline[linecolor=hzerocolor](3.00,1.00)(3.00,2.00)
\psline[linecolor=pcolor]{->}(3.00,1.00)(2.30,1.00)
\pscircle*[linecolor=dotcolor](3.00,1.00){\cirrad}
\psline[linecolor=hzerocolor](3.00,2.00)(3.00,3.00)
\pscircle*[linecolor=dotcolor](3.00,2.00){\cirrad}
\pscircle*[linecolor=dotcolor](3.00,3.00){\cirrad}

\uput{\labelrad}[-90](7,1){$\frac{\gamma}{\tau}h_3$}
\psline[linecolor=NCcolor](7,1)(7,4)
\psline[linecolor=NCcolor](7,1)(7.30,1.30)
\psline[linecolor=NCcolor](7,4)(7.3,4)
\pscircle*[linecolor=NCcolor](7,1){\cirrad}
\pscircle*[linecolor=NCcolor](7,2){\cirrad}
\pscircle*[linecolor=NCcolor](7,3){\cirrad}
\pscircle*[linecolor=NCcolor](7,4){\cirrad}

\end{pspicture}
\end{figure}

\begin{figure}
\caption{$C_2$-equivariant Adams $E_{\infty}$-page in coweight $2$
\label{fig:Adams-2}}
\begin{pspicture}(-1,-1)(9,6)
\psgrid[unit=1,gridcolor=gridline,subgriddiv=0,gridlabelcolor=white](0,0)(9,5)
\tiny
\rput(0,-1){0}
\rput(2,-1){2}
\rput(4,-1){4}
\rput(6,-1){6}
\rput(8,-1){8}

\rput(-0.5,0){0}
\rput(-0.5,2){2}
\rput(-0.5,4){4}

\uput{\labelrad}[-90](0.00,1.00){$\tau^{2} h_0$}
\psline[linecolor=hzerocolor](0.00,1.00)(0.00,2.00)
\psline[linecolor=honecolor](0.00,1.00)(1.00,2.00)
\pscircle*[linecolor=dotcolor](0.00,1.00){\cirrad}
\psline[linecolor=hzerocolor](0.00,2.00)(0.00,3.00)
\pscircle*[linecolor=dotcolor](0.00,2.00){\cirrad}
\psline[linecolor=hzerocolor]{->}(0.00,3.00)(0.00,3.70)
\pscircle*[linecolor=dotcolor](0.00,3.00){\cirrad}
\uput{\labelrad}[-90](2.00,2.00){$(\tau h_1)^2$}
\psline[linecolor=hzerocolor](2.00,2.00)(2.00,3.00)
\psline[linecolor=pcolor](2.00,2.00)(1.00,2.00)
\psline[linecolor=honecolor](2.00,2.00)(3.00,3.00)
\pscircle*[linecolor=dotcolor](2.00,2.00){\cirrad}
\psline[linecolor=honecolor](1.00,2.00)(2.00,3.00)
\pscircle*[linecolor=dotcolor](1.00,2.00){\cirrad}
\psline[linecolor=pcolor](3.00,3.00)(2.00,3.00)
\pscircle*[linecolor=dotcolor](3.00,3.00){\cirrad}
\pscircle*[linecolor=dotcolor](2.00,3.00){\cirrad}
\uput{\labelrad}[-90](6.00,2.00){$h_2^{2}$}
\psline[linecolor=pcolor]{->}(6.00,2.00)(5.30,2.00)
\pscircle*[linecolor=dotcolor](6.00,2.00){\cirrad}

\uput{\labelrad}[-90](9,5){$\frac{\gamma}{\tau}P h_1$}
\psline[linecolor=NCcolor](9,5)(9.3,5)
\psline[linecolor=NCcolor](9,5)(9.30,5.30)
\pscircle*[linecolor=NCcolor](9,5){\cirrad}
\end{pspicture}
\end{figure}
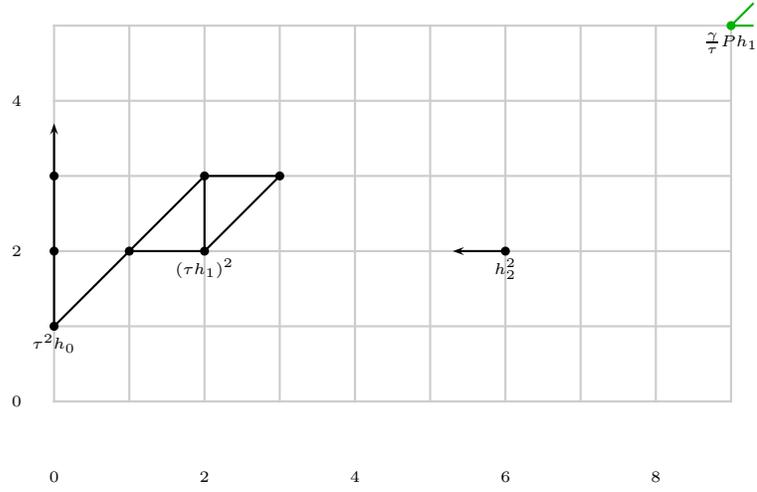

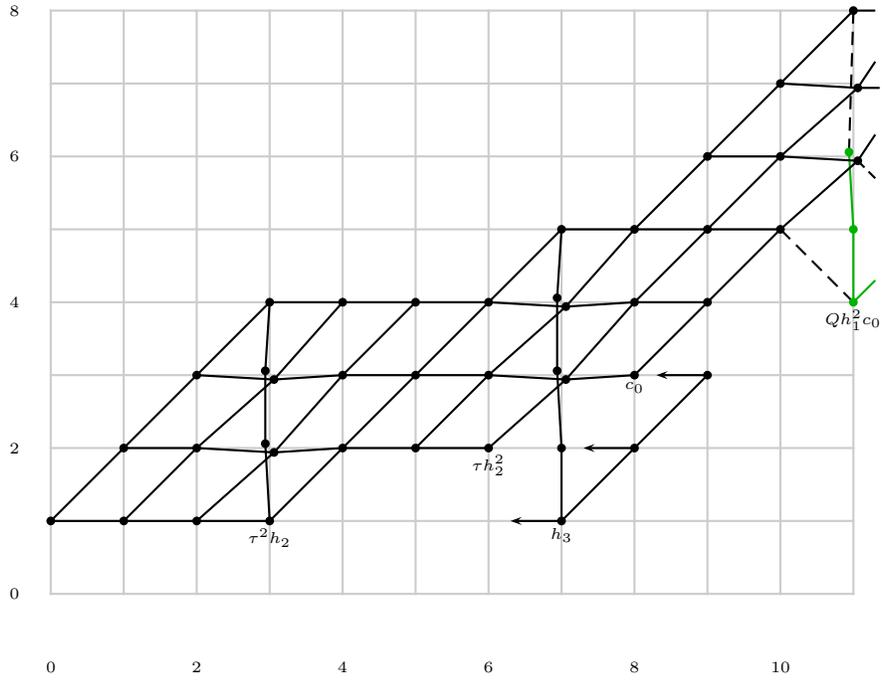
\begin{figure}
\caption{$C_2$-equivariant Adams $E_{\infty}$-page in coweight $3$
\label{fig:Adams-3}}

\psset{unit=0.97cm}
\begin{pspicture}(-1,-1)(12,9)
\psgrid[unit=1,gridcolor=gridline,subgriddiv=0,gridlabelcolor=white](0,0)(11,8)
\tiny
\rput(0,-1){0}
\rput(2,-1){2}
\rput(4,-1){4}
\rput(6,-1){6}
\rput(8,-1){8}
\rput(10,-1){10}

\rput(-0.5,0){0}
\rput(-0.5,2){2}
\rput(-0.5,4){4}
\rput(-0.5,6){6}
\rput(-0.5,8){8}

\uput{\labelrad}[-90](11,4){$Q h_1^2 c_0$}
\psline[linecolor=pcolor,linestyle=dashed](11,4)(10,5)
\psline[linecolor=pcolor,linestyle=dashed](11.06,5.94)(11.3,5.7)
\psline[linecolor=NCcolor](11,4)(11.30,4.30)
\psline[linecolor=NCcolor](11,4)(11,5)
\psline[linecolor=NCcolor](11,5)(10.94,6.06)
\psline[linecolor=hzerocolor,linestyle=dashed](10.94,6.06)(11,8)
\pscircle*[linecolor=NCcolor](11,4){\cirrad}
\pscircle*[linecolor=NCcolor](11,5){\cirrad}
\pscircle*[linecolor=NCcolor](10.94,6.06){\cirrad}

\uput{\labelrad}[-90](3.00,1.00){$\tau^{2} h_2$}
\psline[linecolor=hzerocolor](3.00,1.00)(2.94,2.06)
\psline[linecolor=pcolor](3.00,1.00)(2.00,1.00)
\psline[linecolor=honecolor](3.00,1.00)(4.00,2.00)
\pscircle*[linecolor=dotcolor](3.00,1.00){\cirrad}
\psline[linecolor=pcolor](2.00,1.00)(1.00,1.00)
\psline[linecolor=honecolor](2.00,1.00)(3.06,1.94)
\pscircle*[linecolor=dotcolor](2.00,1.00){\cirrad}
\psline[linecolor=pcolor](1.00,1.00)(0.00,1.00)
\psline[linecolor=honecolor](1.00,1.00)(2.00,2.00)
\pscircle*[linecolor=dotcolor](1.00,1.00){\cirrad}
\psline[linecolor=honecolor](0.00,1.00)(1.00,2.00)
\pscircle*[linecolor=dotcolor](0.00,1.00){\cirrad}
\psline[linecolor=hzerocolor](2.94,2.06)(2.94,3.06)
\pscircle*[linecolor=dotcolor](2.94,2.06){\cirrad}
\psline[linecolor=hzerocolor](2.94,3.06)(3.00,4.00)
\pscircle*[linecolor=dotcolor](2.94,3.06){\cirrad}
\uput{\labelrad}[-90](6.00,2.00){$\tau h_2^{2}$}
\psline[linecolor=pcolor](6.00,2.00)(5.00,2.00)
\psline[linecolor=honecolor](6.00,2.00)(7.06,2.94)
\pscircle*[linecolor=dotcolor](6.00,2.00){\cirrad}
\psline[linecolor=pcolor](5.00,2.00)(4.00,2.00)
\psline[linecolor=honecolor](5.00,2.00)(6.00,3.00)
\pscircle*[linecolor=dotcolor](5.00,2.00){\cirrad}
\psline[linecolor=pcolor](4.00,2.00)(3.06,1.94)
\psline[linecolor=honecolor](4.00,2.00)(5.00,3.00)
\pscircle*[linecolor=dotcolor](4.00,2.00){\cirrad}
\psline[linecolor=pcolor](3.06,1.94)(2.00,2.00)
\psline[linecolor=honecolor](3.06,1.94)(4.00,3.00)
\pscircle*[linecolor=dotcolor](3.06,1.94){\cirrad}
\psline[linecolor=pcolor](2.00,2.00)(1.00,2.00)
\psline[linecolor=honecolor](2.00,2.00)(3.06,2.94)
\pscircle*[linecolor=dotcolor](2.00,2.00){\cirrad}
\psline[linecolor=honecolor](1.00,2.00)(2.00,3.00)
\pscircle*[linecolor=dotcolor](1.00,2.00){\cirrad}
\uput{\labelrad}[-90](7.00,1.00){$h_3$}
\psline[linecolor=hzerocolor](7.00,1.00)(7.00,2.00)
\psline[linecolor=pcolor]{->}(7.00,1.00)(6.30,1.00)
\psline[linecolor=honecolor](7.00,1.00)(8.00,2.00)
\pscircle*[linecolor=dotcolor](7.00,1.00){\cirrad}
\psline[linecolor=hzerocolor](7.00,2.00)(6.94,3.06)
\pscircle*[linecolor=dotcolor](7.00,2.00){\cirrad}
\psline[linecolor=hzerocolor](6.94,3.06)(6.94,4.06)
\pscircle*[linecolor=dotcolor](6.94,3.06){\cirrad}
\psline[linecolor=hzerocolor](6.94,4.06)(7.00,5.00)
\pscircle*[linecolor=dotcolor](6.94,4.06){\cirrad}
\psline[linecolor=pcolor]{->}(8.00,2.00)(7.30,2.00)
\psline[linecolor=honecolor](8.00,2.00)(9.00,3.00)
\pscircle*[linecolor=dotcolor](8.00,2.00){\cirrad}
\uput{\labelrad}[-90](8.00,3.00){$c_0$}
\psline[linecolor=pcolor](8.00,3.00)(7.06,2.94)
\psline[linecolor=honecolor](8.00,3.00)(9.00,4.00)
\pscircle*[linecolor=dotcolor](8.00,3.00){\cirrad}
\psline[linecolor=pcolor](7.06,2.94)(6.00,3.00)
\psline[linecolor=honecolor](7.06,2.94)(8.00,4.00)
\pscircle*[linecolor=dotcolor](7.06,2.94){\cirrad}
\psline[linecolor=pcolor](6.00,3.00)(5.00,3.00)
\psline[linecolor=honecolor](6.00,3.00)(7.06,3.94)
\pscircle*[linecolor=dotcolor](6.00,3.00){\cirrad}
\psline[linecolor=pcolor](5.00,3.00)(4.00,3.00)
\psline[linecolor=honecolor](5.00,3.00)(6.00,4.00)
\pscircle*[linecolor=dotcolor](5.00,3.00){\cirrad}
\psline[linecolor=pcolor](4.00,3.00)(3.06,2.94)
\psline[linecolor=honecolor](4.00,3.00)(5.00,4.00)
\pscircle*[linecolor=dotcolor](4.00,3.00){\cirrad}
\psline[linecolor=pcolor](3.06,2.94)(2.00,3.00)
\psline[linecolor=honecolor](3.06,2.94)(4.00,4.00)
\pscircle*[linecolor=dotcolor](3.06,2.94){\cirrad}
\psline[linecolor=honecolor](2.00,3.00)(3.00,4.00)
\pscircle*[linecolor=dotcolor](2.00,3.00){\cirrad}
\psline[linecolor=pcolor]{->}(9.00,3.00)(8.30,3.00)
\pscircle*[linecolor=dotcolor](9.00,3.00){\cirrad}
\psline[linecolor=pcolor](9.00,4.00)(8.00,4.00)
\psline[linecolor=honecolor](9.00,4.00)(10.00,5.00)
\pscircle*[linecolor=dotcolor](9.00,4.00){\cirrad}
\psline[linecolor=pcolor](8.00,4.00)(7.06,3.94)
\psline[linecolor=honecolor](8.00,4.00)(9.00,5.00)
\pscircle*[linecolor=dotcolor](8.00,4.00){\cirrad}
\psline[linecolor=pcolor](7.06,3.94)(6.00,4.00)
\psline[linecolor=honecolor](7.06,3.94)(8.00,5.00)
\pscircle*[linecolor=dotcolor](7.06,3.94){\cirrad}
\psline[linecolor=pcolor](6.00,4.00)(5.00,4.00)
\psline[linecolor=honecolor](6.00,4.00)(7.00,5.00)
\pscircle*[linecolor=dotcolor](6.00,4.00){\cirrad}
\psline[linecolor=pcolor](5.00,4.00)(4.00,4.00)
\pscircle*[linecolor=dotcolor](5.00,4.00){\cirrad}
\psline[linecolor=pcolor](4.00,4.00)(3.00,4.00)
\pscircle*[linecolor=dotcolor](4.00,4.00){\cirrad}
\pscircle*[linecolor=dotcolor](3.00,4.00){\cirrad}
\psline[linecolor=pcolor](10.00,5.00)(9.00,5.00)
\psline[linecolor=honecolor](10.00,5.00)(11.06,5.94)
\pscircle*[linecolor=dotcolor](10.00,5.00){\cirrad}
\psline[linecolor=pcolor](9.00,5.00)(8.00,5.00)
\psline[linecolor=honecolor](9.00,5.00)(10.00,6.00)
\pscircle*[linecolor=dotcolor](9.00,5.00){\cirrad}
\psline[linecolor=pcolor](8.00,5.00)(7.00,5.00)
\psline[linecolor=honecolor](8.00,5.00)(9.00,6.00)
\pscircle*[linecolor=dotcolor](8.00,5.00){\cirrad}
\pscircle*[linecolor=dotcolor](7.00,5.00){\cirrad}
\psline[linecolor=pcolor](11.06,5.94)(10.00,6.00)
\psline[linecolor=honecolor](11.06,5.94)(11.30,6.30)
\pscircle*[linecolor=dotcolor](11.06,5.94){\cirrad}
\psline[linecolor=pcolor](10.00,6.00)(9.00,6.00)
\psline[linecolor=honecolor](10.00,6.00)(11.06,6.94)
\pscircle*[linecolor=dotcolor](10.00,6.00){\cirrad}
\psline[linecolor=honecolor](9.00,6.00)(10,7)
\pscircle*[linecolor=dotcolor](9.00,6.00){\cirrad}
\psline[linecolor=honecolor](10,7)(11,8)
\psline[linecolor=pcolor](11.06,6.94)(10,7)
\pscircle*[linecolor=dotcolor](10,7){\cirrad}
\psline[linecolor=honecolor](11.06,6.94)(11.30,7.30)
\psline[linecolor=pcolor](11.06,6.94)(11.36,6.94)
\pscircle*[linecolor=dotcolor](11.06,6.94){\cirrad}
\psline[linecolor=pcolor](11.00,8.00)(11.3,8)
\pscircle*[linecolor=dotcolor](11,8){\cirrad}
\end{pspicture}
\end{figure}

Figure \ref{fig:pi-structure} describes some of the global structure
of $\pi^{C_2}_{*,*}$ in graphical form;
see \cite{DI17b}*{Section 1.2} for more discussion.
The groups $\pi^{C_2}_{s,w}$ can be separated into different regions, with
qualitatively different behavior in each region:
\begin{itemize}
\item
zero: The groups in this region are all zero.
\item
$\R$-motivic: The groups in this region are isomorphic to
$\pi^\R_{*,*}$, according to Theorem \ref{thm:pi-iso}.
\item
$\tau$-periodic: The groups in this region display a certain type of
periodicity.  In particular, they can be deduced from groups in the 
$\R$-motivic region.
\item
?: The groups in this region are more complicated, with the occurrence
of purely equivariant phenomena.
\end{itemize}

\begin{figure}
\caption{The structure of $\pi_{s,w}^{C_2}$
\label{fig:pi-structure}}
\psset{unit=0.5cm}
\begin{pspicture}(-4,-4)(10,11)
\psline[linecolor=mygray]{->}(-4,0)(10,0)
\psline[linecolor=mygray](2,-0.2)(2,0.2)
\rput(2,-0.4){\tiny $2$}
\rput(10,-0.4){$s$}
\psline[linecolor=mygray]{->}(0,-4)(0,10)
\psline[linecolor=mygray](-0.2,3)(0.2,3)
\rput(-0.4,3){\tiny $3$}

\rput(-0.4,10){$w$}

\psline{->}(0,0)(-4,-4)
\rput(-2,4){zero}

\psline{->}(0,0)(0,10)
\psline{->}(0,1)(9,10)
\psline{->}(2,3)(10,7)

\rput(2.5,7){$\tau$-periodic}
\rput(6,2){$\R$-motivic}
\rput(6.5,6.3){?}

\end{pspicture}
\end{figure}
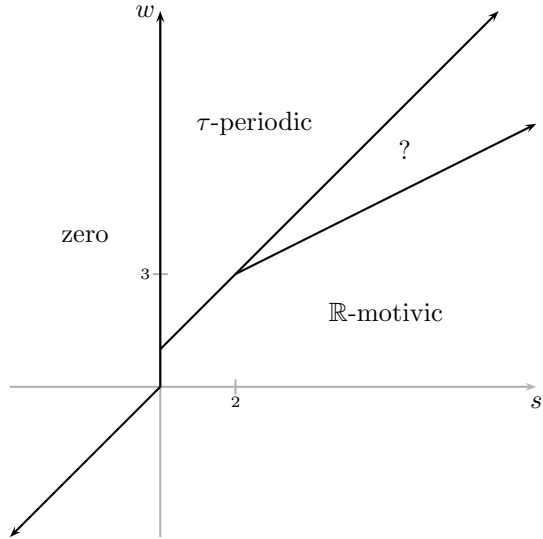

Our result is sharp in the following sense.
In Example \ref{ex:sharp},
we will describe an infinite family of elements in
$\pi^{C_2}_{*,*}$ lying just outside the range under consideration
that are not in the image of Betti realization.

We have at least two motivations for proving Theorem \ref{thm:pi-iso}.
First, this theorem is a self-evidently useful tool in the on-going 
program to carry out explicit computations of motivic and equivariant
stable homotopy groups.
Second, the theorem is used in \cite{BI19} to compute 
some classical Mahowald invariants from detailed information
about $\R$-motivic stable homotopy groups.

\subsection{Notation}
We write $\Ext_\C$ (resp., $\Ext_\R$, $\Ext_{C_2}$)
for the $\C$-motivic (resp., $\R$-motivic, $C_2$-equivariant)
$\Ext$ groups that serve as the $E_2$ page of the
$\C$-motivic (resp., $\R$-motivic, $C_2$-equivariant)
Adams spectral sequence.
We grade these $\Ext$ groups in the form $(s,f,w)$, where
$s$ is the stem (i.e., the total degree minus the homological degree),
$f$ is the Adams filtration (i.e., the homological degree), and
$w$ is the (motivic or equivariant) weight.

\section{The $\C$-motivic cofiber of $\tau$}

We recall some needed facts from $\C$-motivic stable homotopy theory
\cite{Isaksen14c} \cite{IWX20}.

Let $S/\tau$ be the cofiber of $\tau$ in the 2-complete 
$\C$-motivic stable homotopy
category, and let $\Ext_\C(S/\tau)$ be the  $E_2$-page  of the Adams 
spectral sequence that converges
to the homotopy groups of $S/\tau$.
The cofiber sequence
\[
\xymatrix@1{
S^{0,-1} \ar[r]^\tau & S^{0,0} \ar[r]^i & S/\tau \ar[r]^p & S^{1,-1}
}
\]
induces a long exact sequence
\begin{equation}
\label{eq:S/tau}
\xymatrix@1@-2pt{
\dots\ar[r] & \Ext_\C^{s,f,w+1} \ar[r]^\tau & \Ext_\C^{s,f,w} \ar[r]^-i &
\Ext_\C^{s,f,w}(S/\tau) \ar[r]^-p &
\Ext_\C^{s-1,f+1,w+1} \ar[r] & \dots.
}
\end{equation}

\begin{lemma}
\label{lem:S/tau-vanish}
Let $2w - s < 1$.  The group $\Ext_\C^{s,f,w}(S/\tau)$ is:
\begin{enumerate}
\item 
a copy of $\F_2$, generated by $i(h_0^f)$, if $s = 0$ and $w = 0$.
\item
zero otherwise.
\end{enumerate}
\end{lemma}

\begin{proof}
The algebraic Novikov spectral sequence
converges to the Adams-Novikov $E_2$-page 
$\Ext_{BP_* BP}(BP_*, BP_*)$ \cite{Miller75} \cite{Novikov67}. 
The group $\Ext_\C^{s,f,w}(S/\tau)$ is isomorphic to 
a part of the algebraic Novikov $E_2$-page that contributes to 
the Adams-Novikov $E_2$-page in stem $s$ and filtration
$2w-s$ \cite{GWX}*{Theorem 1.14}.
The result follows from an elementary analysis of the
algebraic Novikov $E_2$-page in Adams-Novikov filtration zero.
\end{proof}

The vanishing result of Lemma \ref{lem:S/tau-vanish} can be applied to
the long exact sequence (\ref{eq:S/tau}) to obtain information
about multiplication by $\tau$ on $\Ext_\C$.

\begin{prop}
\label{prop:Ext_C-tau}
Let $x$ be a non-zero element of $\Ext_\C$ of degree $(s,f,w)$.
\begin{enumerate}
\item
If $x$ is not divisible by $\tau$, then $2w - s \geq 1$, 
or $x = h_0^f$.
\item
If $x$ is annihilated by $\tau$, then $2w - s \geq 4$.
\end{enumerate}
\end{prop}

\begin{proof}
For part (1), the image
$i(x)$ of $x$ in $\Ext_\C(S/\tau)$ must be non-zero.
By Lemma \ref{lem:S/tau-vanish}, if $2w-s$ is less than 1, then
$x$ must be $h_0^f$.

For part (2), $x$ must lie in the image of $p$.
The pre-image of $x$ has degree $(s+1, f-1, w-1)$.
Lemma \ref{lem:S/tau-vanish} implies that
\begin{align*}
2(w-1) - (s+1) & = 2w - s - 3\geq 1. \qedhere
\end{align*}
\end{proof}

\section{The $\rho$-Bockstein spectral sequence}
\label{sctn:Bockstein}

Recall from \cite{GHIR20}*{Section 2} that $\Ext_{C_2}$ splits
as $\Ext_\R \oplus \Ext_{NC}$, where $\Ext_{NC}$ is associated to
the ``negative cone" in the $C_2$-equivariant cohomology of a point.
Betti realization induces the natural inclusion
\begin{equation}
\label{eq:split}
\Ext_\R \map \Ext_\R \oplus \Ext_{NC}.
\end{equation}
In order to obtain an isomorphism $\Ext_\R \map \Ext_{C_2}$ in a range
of degrees,
we must show that $\Ext_{NC}$ vanishes in that range.

The groups $\Ext_{NC}$ can be computed by a $\rho$-Bockstein spectral
sequence, denoted $E^-$ in \cite{GHIR20}.
The $E^-_1$-page of this spectral sequence contains elements of two
types.

First, there are elements of the form
$\nc{a}{b} x$, where $0 \leq a$,
$1 \leq b$, and
$x$ is an element
of $\Ext_{\C}$ that is $\tau$-free and not divisible by $\tau$.
If $x$ has degree $(s, f, w)$ in $\Ext_\C$, then
$\nc{a}{b} x$ has degree
$(s + a, f, w + a + b + 1)$.

The second type of element in $E^-_1$ is of the form
$\Qc{a}{b} x$, where $0 \leq a$, $0 \leq b \leq k$, and
$x$ is an element of $\Ext_{\C}$ that is annihilated by $\tau$
and is divisible by $\tau^k$ but not by $\tau^{k+1}$.
If $x$ has degree $(s, f, w)$ in $\Ext_\C$,
then $\Qc{a}{b} x$ has degree
$(s + a + 1, f - 1, w + a + b + 1)$.

\begin{lemma}
\label{lem:Eminus-vanish}
\mbox{}
\begin{enumerate}
\item
If $\nc{a}{b} x$ 
is a non-zero element of $E^-_1$ with degree $(s, f, w)$, then
$x = h_0^f$ or $2w - s \geq a + 2b + 3$.
\item
If $\Qc{a}{b} x$ 
is a non-zero element of $E^-_1$ with degree $(s, f, w)$, 
then 
\[
2w - s \geq a + 2b + 5.
\]
\end{enumerate}
\end{lemma}

\begin{proof}
For part (1),
the element $x$ has degree $(s - a, f, w - a - b - 1)$.
Since $x$ is not divisible by $\tau$,
Proposition \ref{prop:Ext_C-tau} implies that
$x = h_0^f$ or 
\[
2(w - a - b - 1) - (s - a) = 2w - s - a - 2b - 2 \geq 1.
\]

For part (2), 
the element $x$ has degree $(s - a - 1, f + 1, w - a - b - 1)$.
Since $x$ is annihilated by $\tau$,
Proposition \ref{prop:Ext_C-tau} implies that
\begin{align*}
2(w - a - b - 1) - (s - a - 1) & = 2w - s - a - 2b - 1 \geq 4. \qedhere
\end{align*}
\end{proof}

\begin{lemma}
\label{lem:perm-cycles}
Amongst elements of the form $\nc{a}{b} h_0^f$ in $E^-_1$, the only
non-zero permanent cycles are $\frac{\gamma}{\tau^{2k+1}} h_0^f$
in degree $(0, f, 2k+2)$ for all $k \geq 0$.
\end{lemma}

\begin{proof}
As in \cite{GHIR20}*{Proposition 7.7} or \cite{GI19}*{Lemma 4.1}, 
there are Bockstein differentials
\begin{align*}
d_1\left( \frac{\gamma}{\rho \tau^{2k+1}} h_0^f \right) & = 
\frac{\gamma}{\tau^{2k+2}} h_0^{f+1}. \qedhere
\end{align*}
\end{proof}

\begin{prop}
\label{prop:ExtNC}
Let $y$ be a non-zero element of $\Ext_{NC}$ of degree 
$(s, f, w)$.  Then $y$ equals $\frac{\gamma}{\tau} h_0^f$,
or $2w - s \geq 5$.
\end{prop}

\begin{proof}
The element $y$ is represented by an element of the 
$\rho$-Bockstein $E^-_1$-page.
If $y$ is of the form $\nc{a}{b} x$,
then $0 \leq a$ and $1 \leq b$, so 
Lemma \ref{lem:perm-cycles} and 
part (1) of
Lemma \ref{lem:Eminus-vanish} gives the desired result.

On the other hand, if $y$ is of the form
$\Qc{a}{b} x$, then
$0 \leq a$ and $0 \leq b$, so part (2) of
Lemma \ref{lem:Eminus-vanish} gives the desired result.
\end{proof}

\begin{thm}
\label{thm:Ext-iso}
Betti realization 
$\Ext_\R^{s,f,w} \map \Ext_{C_2}^{s,f,w}$ is:
\begin{enumerate}
\item
an injection in all degrees.
\item
an isomorphism if $2w - s < 5$, except when $s = 0$ and $w = 2$.
\end{enumerate}
\end{thm}

\begin{proof}
Diagram (\ref{eq:split}) shows that Betti realization is an injection
and  induces an isomorphism if and only if
$\Ext_{NC}$ vanishes.  Proposition \ref{prop:ExtNC} provides the needed
vanishing result for $\Ext_{NC}$, since
$\frac{\gamma}{\tau} h_0^f$ has degree
$(0, f, 2)$.
\end{proof}

\section{The Adams spectral sequence}

Theorem \ref{thm:Ext-iso} shows that the
$\R$-motivic and $C_2$-equivariant Adams $E_2$-pages are 
isomorphic in a range.  Now we will extend this isomorphism to
higher Adams pages and then to stable homotopy groups.
We write $E^\R_r(s, f, w)$ and $E^{C_2}_r(s,f,w)$ for the
$\R$-motivic and $C_2$-equivariant Adams $E_r$-pages in degree
$(s,f,w)$ respectively.

\begin{lemma}
\label{lem:Adams-perm-cycles}
In the $C_2$-equivariant Adams spectral sequence,
the element $\frac{\gamma}{\tau} h_0^f$ is a
permanent cycle.
\end{lemma}

\begin{proof}
Targets of possible differentials on 
$\frac{\gamma}{\tau} h_0^f$ lie in degrees $(-1, r + f, 2)$.
The Adams $E_2$-page is zero in those degrees, as
$\Ext_\R$ vanishes when the coweight $s-w$ is negative 
and $\Ext_{NC}$ vanishes when the stem $s$ is negative.
\end{proof}

\begin{prop}
\label{prop:Adams-iso}
Let $r \geq 2$, or let $r = \infty$.  The Betti realization map
\[
E^\R_r (s, f, w) \map E^{C_2}_r(s,f,w)
\]
\begin{enumerate}
\item
is an isomorphism if $2w-s < 5$ and $(s,w) \neq (0,2)$.
\item
is an injection if $2w-s = 5$.
\end{enumerate}
\end{prop}

\begin{proof}
The proof is by induction on $r$.  The base case $r = 2$ is established
in Theorem \ref{thm:Ext-iso}.
For the sake of induction, assume that the result is known for $r$.
Consider the diagram
\[
\xymatrix{
E^\R_r(s+1,f-r,w) \ar[d] \ar[r]^-{d_r} & 
E^\R_r(s,f,w) \ar[d] \ar[r]^-{d_r} & 
E^\R_r(s-1,f+r,w) \ar[d] \\
E^{C_2}_r(s+1,f-r,w) \ar[r]_-{d_r} &
E^{C_2}_r(s,f,w) \ar[r]_-{d_r} &
E^{C_2}_r(s-1,f+r,w).
}
\]

For the induction step in part (1), 
suppose that $2w-s < 5$, and that $(s,w) \neq (0,2)$.
Then the induction assumption implies that
the left and middle vertical arrows are isomorphisms, while the
right vertical arrow is an injection.  A standard diagram chase
implies that $E^\R_{r+1}(s,f,w) \map E^{C_2}_{r+1}(s,f,w)$ is an
isomorphism.

The induction step for part (2) splits into two cases.
Suppose that $2w-s<6$ and that $(s,w) \neq (-1,2)$.
The induction assumption implies that the
left vertical arrow is an isomorphism and the middle vertical
arrow is an injection (and nothing can be said about the right vertical
arrow).  Again, a diagram chase shows that 
$E^\R_{r+1}(s,f,w) \map E^{C_2}_{r+1}(s,f,w)$ is an injection.

Now suppose that $(s,w) = (-1,2)$.
In this case, the left vertical arrow is known only to be an injection.  
However, Lemma \ref{lem:Adams-perm-cycles}
implies that this doesn't matter, and the same diagram chase gives 
the desired conclusion.
This finishes the induction step for part (1).

Finally, the case $r = \infty$ follows from the previous cases, since 
$E^\R_\infty(s,f,w)$ and $E^{C_2}_\infty(s,f,w)$ are equal to
$E^\R_r(s,f,w)$ and $E^{C_2}_r(s,f,w)$ for $r > N$, where $N$ depends
on $(s,f,w)$.
\end{proof}

\begin{ex}
\label{ex:sharp}
Consider the elements $\frac{\gamma}{\tau} P^k h_1$ in degree
$(8k + 1, 4k + 1, 4k + 3)$.  Note that
\[
2(4k+3) - (8k + 1) = 5,
\]
so these elements lie just outside the range in part (1) of
Theorem \ref{thm:pi-iso}.
These elements are permanent cycles in both the 
$\rho$-Bockstein and Adams spectral sequences, since they lie 
near the top of the Adams chart and there are no possible elements
to serve as targets for differentials.
Moreover, they are not hit by any $\rho$-Bockstein or Adams
differentials since they are detected by the equivariant
spectrum $ko_{C_2}$ \cite{GHIR20}.

Therefore, for all $k \geq 1$,
\[
\pi^{\R}_{8k+1, 4k+3} \map \pi^{C_2}_{8k+1,4k+3}
\]
is not an isomorphism.  Thus,
part (1) of Theorem \ref{thm:pi-iso} is sharp, in the sense that
there is no larger range of degrees bounded by linear inequalities
in which Betti realization is an isomorphism.
\end{ex}


\begin{bibdiv}
\begin{biblist}

\bib{BI19}{article}{
	author={Belmont, Eva},
	author={Isaksen, Daniel C.},
	title={$\R$-motivic stable stems},
	status={preprint},
	date={2020},
}

\bib{DI10}{article}{
   author={Dugger, Daniel},
   author={Isaksen, Daniel C.},
   title={The motivic Adams spectral sequence},
   journal={Geom. Topol.},
   volume={14},
   date={2010},
   number={2},
   pages={967--1014},
   issn={1465-3060},
   review={\MR{2629898 (2011e:55024)}},
   doi={10.2140/gt.2010.14.967},
}

\bib{DI17}{article}{
    author={Dugger, Daniel},
	author={Isaksen, Daniel C.},
     title={Low-dimensional {M}ilnor-{W}itt stems over {$\mathbb{R}$}},
   journal={Ann. K-Theory},
    volume={2},
      date={2017},
    number={2},
     pages={175--210},
      issn={2379-1683},
}

\bib{DI17b}{article}{
   author={Dugger, Daniel},
   author={Isaksen, Daniel C.},
   title={$\mathbb{Z}/2$-equivariant and $\mathbb{R}$-motivic stable stems},
   journal={Proc. Amer. Math. Soc.},
   volume={145},
   date={2017},
   number={8},
   pages={3617--3627},
   issn={0002-9939},
   review={\MR{3652813}},
   doi={10.1090/proc/13505},
}

\bib{GWX}{article}{
	author={Gheorghe, Bogdan},
	author={Wang, Guozhen},
	author={Xu, Zhouli},
	title={The special fiber of the motivic deformation of the stable homotopy category is algebraic},
    year={2018},
    status={preprint},
    eprint={arXiv:1809.09290},
}

\bib{GHIR20}{article}{
	author={Guillou, Bertrand J.},
	author={Hill, Michael A.},
	author={Isaksen, Daniel C.},
	author={Ravenel, Douglas Conner},
	title={The cohomology of $C_2$-equivariant $A(1)$ and the homotopy of $ko_{C_2}$},
	journal={Tunisian J. Math.},
	date={2020},
	volume={2},
	number={3},
	pages={567--632},
}

\bib{GI19}{article}{
	author={Guillou, Bertrand J.},
	author={Isaksen, Daniel C.},
	title={The Bredon-Landweber region in $C_2$-equivariant stable homotopy groups},
	status={preprint},
	eprint={arXiv:1907.01539},
}

\bib{HO16}{article}{
   author={Heller, J.},
   author={Ormsby, K.},
   title={Galois equivariance and stable motivic homotopy theory},
   journal={Trans. Amer. Math. Soc.},
   volume={368},
   date={2016},
   number={11},
   pages={8047--8077},
   issn={0002-9947},
   review={\MR{3546793}},
   doi={10.1090/tran6647},
}

\bib{Isaksen14c}{article}{
	author={Isaksen, Daniel C.},
	title={Stable stems},
	journal={Mem. Amer. Math. Soc.},
	status={to appear},
}

\bib{IWX20}{article}{
	author={Isaksen, Daniel C.},
	author={Wang, Guozhen},
	author={Xu, Zhouli},
	title={More stable stems},
	status={preprint},
	date={2020},
}

\bib{Miller75}{book}{
   author={Miller, Haynes Robert},
   title={Some algebraic aspects of the Adams-Novikov spectral sequence},
   note={Thesis (Ph.D.)--Princeton University},
   publisher={ProQuest LLC, Ann Arbor, MI},
   date={1975},
   pages={103},
   review={\MR{2625232}},
}

\bib{Novikov67}{article}{
   author={Novikov, S. P.},
   title={Methods of algebraic topology from the point of view of cobordism theory},
   language={Russian},
   journal={Izv. Akad. Nauk SSSR Ser. Mat.},
   volume={31},
   date={1967},
   pages={855--951},
   issn={0373-2436},
   review={\MR{0221509}},
}

\bib{MV99}{article}{
   author={Morel, Fabien},
   author={Voevodsky, Vladimir},
   title={${\bf A}\sp 1$-homotopy theory of schemes},
   journal={Inst. Hautes \'Etudes Sci. Publ. Math.},
   number={90},
   date={1999},
   pages={45--143 (2001)},
   issn={0073-8301},
   review={\MR{1813224 (2002f:14029)}},
}

\end{biblist}
\end{bibdiv}
\end{document}